\documentclass[11pt,twoside]{article}

\usepackage{amsmath}
\usepackage{amssymb}
\usepackage{amsfonts}
\usepackage{amsthm}

\newtheorem{thm}{Theorem}[section]

\newtheorem{lem}[thm]{Lemma}

\begin{document}

\markboth{D.M. D\~{a}ianu}{Sections and Fr\'{e}chet Polynomials}

\title{\bf SECTION METHOD AND FR\'{E}CHET POLYNOMIALS}

\author{Dan M. D\~{A}IANU\\
 University "Politehnica" of Timisoara}

\thispagestyle{plain}

\maketitle


\textbf{\textit{Tribute to Professor Borislav Crstici on
the 100th anniversary of his birth}}

\begin{abstract}

Using the section method we characterize the solutions $%
f:U\rightarrow Y$ of the following four equations %
\begin{equation*}
\sum\limits_{i=0}^{n}\left( -1\right) ^{n-i}\tbinom{n}{i}f\left( \sqrt[m]{%
u^{m}+iv^{m}}\right) =\left( n!\right) f\left( v\right) \text{, }
\end{equation*}%
\begin{equation*}
\text{\ }f\left( u\right) +\sum\limits_{i=1}^{n+1}\left( -1\right) ^{i}%
\tbinom{n+1}{i}f\left( \sqrt[m]{u^{m}+iv^{m}}\right) =0,
\end{equation*}%
\begin{equation*}
\sum\limits_{i=0}^{n}\left( -1\right) ^{n-i}\tbinom{n}{i}f\left( \arcsin
\left\vert \sin u\sin ^{i}v\right\vert \right) =\left( n!\right) f\left(
v\right) \text{ and }
\end{equation*}%
\begin{equation*}
f\left( u\right) +\sum\limits_{i=1}^{n+1}\left( -1\right) ^{i}\tbinom{n+1}{i%
}f\left( \arcsin \left\vert \sin u\sin ^{i}v\right\vert \right) =0,
\end{equation*}%
 where $m\geq 2$ \ and $n$\ are\ positive integers,$%
\ U\subseteq 
\mathbb{R}
$ \ is a maximally relevant real domain and $\left( Y,+\right) $%
 \ is an $\left( n!\right) $ -divisible Abelian group.

\end{abstract} 

\section{Introduction}

Educated at the Cluj school of functional equations, approximations and
convexity founded by Academician Tiberiu Popoviciu, Professor Borislav Crstici had
among his main concerns the functional characterization of polynomials and
their generalizations. Without pretending exhaustiveness, we mention here
Professor's thesis dedicated to the functional equations that define polynomials 
\cite{bc69} and the additions made in \cite{kc89}, \cite{nc87} and \cite{bc96}. A brief
presentation of Professor's personality is given in \cite{cd14}. In this context we mention the works
\cite{aak}, \cite{al17}, \cite{al23}, \cite{as15}, \cite{as16}, 
\cite{d14}, \cite{dd14}, \cite{dd18}, \cite{d19}, \cite{dm19}, \cite{sz14}
and \cite{tst22}, published only in the last decade and which contain some generalizations and analyses of different types of polynomials.

Let $m,n$ be positive integers, $m\geq 2$ and$\ \left( Y,+\right) $\ be an $%
\left( n!\right) $-divisible Abelian group - i.e. the group homomorphism $%
Y\rightarrow Y$, $y\mapsto \left( n!\right) y$ is an isomorphism. This paper
is dedicated to characterize the solutions $f:%
\mathbb{R}
\rightarrow Y$\ of the equations 
\begin{equation}
\sum\limits_{i=0}^{n}\left( -1\right) ^{n-i}\tbinom{n}{i}f\left( \sqrt[m]{%
u^{m}+iv^{m}}\right) =\left( n!\right) f\left( v\right) \text{ for all }%
u,v\in 
\mathbb{R}
\text{,}  \label{1}
\end{equation}%
\begin{equation}
f\left( u\right) +\sum\limits_{i=1}^{n+1}\left( -1\right) ^{i}\tbinom{n+1}{i%
}f\left( \sqrt[m]{u^{m}+iv^{m}}\right) =0\text{ for all }u,v\in 
\mathbb{R}
\text{,}  \label{2}
\end{equation}%
and the solutions $f:U:=%
\mathbb{R}
\diagdown \left\{ k\pi |\text{ }k\in 
\mathbb{Z}
\right\} \rightarrow Y$\ of the equations%
\begin{equation}
\sum\limits_{i=0}^{n}\left( -1\right) ^{n-i}\tbinom{n}{i}f\left( \arcsin
\left\vert \sin u\sin ^{i}v\right\vert \right) =\left( n!\right) f\left(
v\right) \text{ for all }u,v\in U\text{, }  \label{3}
\end{equation}
\begin{equation}
f\left( u\right) +\sum\limits_{i=1}^{n+1}\left( -1\right) ^{i}\tbinom{n+1}{i%
}f\left( \arcsin \left\vert \sin u\sin ^{i}v\right\vert \right) =0\text{ for
all }u,v\in U\text{.}  \label{4}
\end{equation}%
The tool use for these characterizations is the section method \cite{dm20}, 
\cite{dd22}.

\section{Framework}

Everywhere in what follows $\left( X,+\right) $\ is a commutative semigroup, 
$n$\ is a positive integer and $\left( Y,+\right) $\ is an $\left( n!\right) $%
-divisible Abelian group. We denote by $\mathcal{S}_{i}$\ the set of the
solutions of equation $\left( i\right) $; for instance $\mathcal{S}_{\ref{1}}
$\ is the set of all functions $f:%
\mathbb{R}
\rightarrow Y$\ that satisfy equation (\ref{1}). Let $j$\ be a nonnegative
integer; we will use the operator%
\begin{equation*}
\Delta _{y}^{j}:Y^{X}\rightarrow Y^{X}\text{, }\Delta _{y}^{j}\rho \left(
x\right) :=\sum\limits_{i=0}^{j}\left( -1\right) ^{j-i}\tbinom{j}{i}\rho
\left( x+iy\right) 
\end{equation*}%
for $y\in X$; $\mathcal{M}_{j}\left( X,Y\right) $ denotes the $j$-\textit{%
monomials}, i.e.\ the solutions $\rho :X\rightarrow Y$\ of the equation%
\begin{equation*}
\Delta _{y}^{j}\rho \left( x\right) =\left( j!\right) \rho \left( y\right) 
\text{ for all }x,y\in X
\end{equation*}%
and $\mathcal{P}_{n}\left( X,Y\right) $\ denotes the (\textit{Fr\'{e}chet})%
\textit{\ }$n$\textit{-polynomials}, i.e. the solutions $\rho :X\rightarrow Y
$\ of the equation%
\begin{equation*}
\Delta _{y}^{n+1}\rho \left( x\right) =0\text{ for all }x,y\in X.
\end{equation*}%
The first characterization of continuous real $n$-polynomials by this
equation was realized by Fr\'{e}chet in \cite{f09}. A detailed analysis of Fr%
\'{e}chet polynomials in the present framework was given by Djokovi\'{c} in 
\cite{dj69}; from this paper we will use only the following result.

\begin{lem}
Let $\rho :X\rightarrow Y.$ Then $\rho \in \mathcal{P}_{n}\left( X,Y\right) $%
\ if and only if there exists $\rho _{j}\in \mathcal{M}_{j}\left( X,Y\right) 
$ for all $j\in \left\{ 0,1,...,n\right\} $\ such that $\rho
=\sum\limits_{j=0}^{n}\rho _{j}$.
\end{lem}

The section method \cite{dm20}, \cite{dd22} provides, among other things, a
technique for solving - partial or total - some equations whose solutions
are composite functions. We will give below only a few rudiments of this
method adapted to our goals.

Let $g:U\rightarrow X$\ be a surjection and $g^{\prime }:X\rightarrow U$\ be
a section of $g$ (i.e. $g\circ g^{\prime }=id_{X}$). Let also the functions%
\begin{equation*}
G:Y^{X}\times X^{2}\rightarrow Y\text{, }H:Y\times Y\rightarrow Y
\end{equation*}%
and the equation%
\begin{equation}
G\left( f\circ g^{\prime },\left( g\left( u\right) ,g\left( v\right) \right)
\right) =H\left( f\left( u\right) ,f\left( v\right) \right) \text{ for all }%
u,v\in U,  \label{5}
\end{equation}%
where the unknown is $f:U\rightarrow Y$. The equation%
\begin{equation}
G\left( \rho ,\left( x,y\right) \right) =H\left( \rho \left( x\right) ,\rho
\left( y\right) \right) \text{ for all }x,y\in X,  \label{6}
\end{equation}%
where $\rho :X\rightarrow Y$ is the unknown, is \textit{the characteristic}
of equation (\ref{5}).

We will use the following results extracted from Th. 2.4.1, Th 2.4.2 and Th.
2.6.6 in \cite{dm20}.

\begin{lem}
1. $\left\{ f\circ g^{\prime }|f\in \mathcal{S}_{5}\right\} \subseteq 
\mathcal{S}_{6}.$

2. $\mathcal{S}_{5}^{c}:=\left\{ \rho \circ g|\rho \in \mathcal{S}%
_{6}\right\} \subseteq \mathcal{S}_{5}.$

3. If $f\in \mathcal{S}_{5}$\ and $u_{0}\in U$\ such that the function%
\begin{equation*}
f\left( U\right) \rightarrow Y\text{, }y\mapsto H\left( f\left( u_{0}\right)
,y\right) \text{ (or }y\mapsto H\left( y,f\left( u_{0}\right) \right) \text{)%
}
\end{equation*}%
is one-to-one, then $f\in \mathcal{S}_{5}^{c}$.
\end{lem}

The functions in $\mathcal{S}_{5}^{c}$ are named \textit{canonical solutions 
}(of equation (\ref{5})). Thus Lemma 2.2.2 gives a partial solution for
equation (\ref{5}) and Lemma 2.2.3 provides sufficient conditions under which
a solution of equation (\ref{5}) is a canonical solution.

In the following we will use the notions and the conventions introduced
above.

\section{Radical-Fr\'{e}chet equations}

Let $m\geq 2$\ be an integer, $U=\mathbb{R} $ and $\left( X,+\right) $\ be the additive semigroup
defined by%
\begin{equation*}
\left( X,+\right) :=\left\{ 
\begin{array}{c}
\left( 
\mathbb{R}
,+\right) \text{ if }m\text{\ is odd} \\ 
\left( 
\mathbb{R}
_{+},+\right) \text{ if }m\text{\ is even}%
\end{array}%
,\right. 
\end{equation*}%
where $%
\mathbb{R}
_{+}:=[0,\infty ).$

First we characterize the solutions of the \textit{radical-monomial }%
equation (\ref{1}).

\begin{thm}
Let $f:%
\mathbb{R}
\rightarrow Y$\ be a function. Then $f$\ is a solution of equation (\ref{1})
if and only if there exists an $n$-monomial $\rho \in \mathcal{M}_{n}\left(
X,Y\right) $\ such that%
\begin{equation}
f\left( u\right) =\rho \left( u^{m}\right) \text{ for all }u\in 
\mathbb{R}
\text{.}  \label{7}
\end{equation}
\end{thm}

\begin{proof}
We apply the section method for the surjection%
\begin{equation*}
g:%
\mathbb{R}
\rightarrow X\text{, }u\mapsto u^{m},
\end{equation*}%
its section 
\begin{equation*}
g^{\prime }:X\rightarrow 
\mathbb{R}
\text{, }x\mapsto \sqrt[m]{x},
\end{equation*}%
and the functions%
\begin{equation*}
G:Y^{X}\times X^{2}\rightarrow Y\text{, }G\left( \rho ,\left( x,y\right)
\right) :=\Delta _{y}^{n}\rho \left( x\right) ,
\end{equation*}%
\begin{equation*}
H:Y\times Y\rightarrow Y\text{, \ }H\left( y_{1},y_{2}\right) :=\left(
n!\right) y_{2}.
\end{equation*}%
Then equation (\ref{1}) becomes equation (\ref{5})\
and its characteristic is equation (\ref{6}). %
Therefore the characteristic of equation (\ref{1}) is exactly the $n$%
-monomial equation%
\begin{equation*}
\Delta _{y}^{n}\rho \left( x\right) =\left( n!\right) \rho \left( y\right) 
\text{ for all }x,y\in X
\end{equation*}%
and all its solutions are in $\mathcal{M}_{n}\left( X,Y\right) .$

1. Let $\rho \in \mathcal{M}_{n}\left( X,Y\right) $\ and $f:%
\mathbb{R}
\rightarrow Y$\ defined by (\ref{7}). Then $f=\rho \circ g$\ and, according
to Lemma 2.2.2, $f$\ is a solution of equation (\ref{1}).

2. For proving the converse it suffices to show that $\mathcal{S}%
_{1}\subseteq \mathcal{S}_{1}^{c}$. Let $f$\ be a solution of equation (\ref%
{1}). Since $\left( Y,+\right) $ is $\left( n!\right) $-divisible, the
function%
\begin{equation*}
f\left( 
\mathbb{R}
\right) \rightarrow Y\text{, }y\mapsto H\left( 0,y\right) =\left( n!\right) y
\end{equation*}%
is injective. According to Lemma 2.2.3, $f$\ is a canonical solution of
equation (\ref{1}), i.e. there exists $\rho \in \mathcal{M}_{n}\left(
X,Y\right) $\ such that $f\left( u\right) =\rho \left( u^{m}\right) $ for
all $u\in 
\mathbb{R}
$.\bigskip 
\end{proof}

Now we are in position to characterize the solutions of the \textit{radical-Fr\'{e}chet} equation (\ref{2}).

\begin{thm}
Let $f:%
\mathbb{R}
\rightarrow Y$\ be a function. Then $f$\ is a solution of equation (\ref{2})
if and only if there exists $\rho _{i}\in \mathcal{M}_{i}\left( X,Y\right) $%
\ for all $i\in \left\{ 0,1,...,n\right\} $\ such that%
\begin{equation}
f\left( u\right) =\rho _{0}\left( u^{m}\right) +\rho _{1}\left( u^{m}\right)
+\cdots +\rho _{n}\left( u^{m}\right) \text{ for all }u\in 
\mathbb{R}
\text{.}  \label{8}
\end{equation}
\end{thm}

\begin{proof}
As in the proof of the previous theorem, let%
\begin{equation*}
g:%
\mathbb{R}
\rightarrow X\text{, }g\left( u\right) :=u^{m}\text{ and }g^{\prime
}:X\rightarrow 
\mathbb{R}
\text{, }g^{\prime
}\left( x\right) :=\sqrt[m]{x}.
\end{equation*}%
Let also be the functions%
\begin{equation*}
G:Y^{X}\times X^{2}\rightarrow Y\text{, }G\left( \rho ,\left( x,y\right)
\right) :=\sum\limits_{i=1}^{n+1}\left( -1\right) ^{i}\tbinom{n+1}{i}\rho
\left( x+iy\right) \text{ and}
\end{equation*}%
\begin{equation*}
H:Y\times Y\rightarrow Y\text{, \ }H\left( y_{1},y_{2}\right) :=-y_{1}.
\end{equation*}%
Then equation (\ref{2}) can be written in the form (\ref{5}) and its characteristic is (\ref{6}) or, equivalent, the Fr\'{e}chet polynomial equation\ 
\begin{equation*}
\Delta _{y}^{n+1}\rho \left( x\right) =0\text{ for all }x,y\in X.
\end{equation*}%
We note that the solutions of the last equation are given by $\mathcal{P}%
_{n}\left( X,Y\right) $ and their characterization is given by Lemma 2.1.

1. Let $\rho _{i}\in \mathcal{M}_{i}\left( X,Y\right) $\ for all $i\in
\left\{ 0,1,...,n\right\} $\ and $f:%
\mathbb{R}
\rightarrow Y$\ defined by (\ref{8}). Then $\rho
:=\sum\limits_{i=0}^{n}\rho _{i}\in \mathcal{P}_{n}\left( X,Y\right) $\ (by
Lemma 2.1) and $f=\rho \circ g\in \mathcal{S}_{2}$\ (by Lemma 2.2.2).

2. Let $f\in \mathcal{S}_{2}$. To show that $f$\ can be expressed by (\ref{8}%
) with $\rho _{i}\in \mathcal{M}_{i}\left( X,Y\right) $ - or, equivalent,
that $f=\rho \circ g$, where $\rho :=\sum\limits_{i=0}^{n}\rho _{i}\in 
\mathcal{P}_{n}\left( X,Y\right) $ -, it is sufficient to prove that $f$\ is
a canonical solution of equation (\ref{2}). But the function%
\begin{equation*}
f\left( 
\mathbb{R}
\right) \rightarrow Y\text{, }y\mapsto H\left( y,0\right) =-y
\end{equation*}%
is injective; according to Lemma 2.2.3, $f\in \mathcal{S}_{2}^{c}$, and the
theorem is completely proved.
\end{proof}

\section{Arcsine-Fr\'{e}chet equations}

Before proceeding to the characterizations of the solutions of the \textit{arcsine-Fr%
\'{e}chet} equations (\ref{3}) and (\ref{4}), let us note that $U:=%
\mathbb{R}
\diagdown \left\{ k\pi |k\in 
\mathbb{Z}
\right\} $\ is the maximal domain on which these equations has nontrivial
solutions; indeed if there is $k\in 
\mathbb{Z}
$\ such that $k\pi $\ is in the domain, for $u=k\pi $\ in (\ref{3}) we get $%
0=\left( n!\right) f\left( v\right) $, and, since $\left( Y,+\right) $\ is $%
\left( n!\right) $-divisible we immediately obtain $f=0$, i.e. $\mathcal{S}%
_{3}=\left\{ 0\right\} $; analogously, for $v=k\pi $\ in (\ref{4}) we get $%
f=0$ and $\mathcal{S}_{4}=\left\{ 0\right\} $.

In the following lines, the set $X:=(-\infty ,0]$ is endowed with the
addition of real numbers, hence $\left( X,+\right) $ is an Abelian semigroup.

\begin{thm}
Let $f:U:=%
\mathbb{R}
\diagdown \left\{ k\pi |k\in 
\mathbb{Z}
\right\} \rightarrow Y$. Then $f\in \mathcal{S}_{3}$\ if and only if there
exists an $n$-monomial $\rho \in \mathcal{M}_{n}\left( X,Y\right) $\ and%
\begin{equation}
f\left( u\right) =\rho \left( \ln \left\vert \sin u\right\vert \right) \text{
for all }u\in U\text{.}  \label{9}
\end{equation}
\end{thm}

\begin{proof}
We apply\ the section method for 
\begin{equation*}
g:U\rightarrow X\text{, }g\left( u\right) :=\ln \left\vert \sin u\right\vert 
\text{ for all }u\in U\text{,}
\end{equation*}%
its section%
\begin{equation*}
g^{\prime }:X\rightarrow U\text{, }g^{\prime }\left( x\right) :=\arcsin e^{x}%
\text{ for all }x\in X\text{,}
\end{equation*}%
and the functions $G,H$\ defined by%
\begin{equation*}
G:Y^{X}\times X^{2}\rightarrow Y\text{, }G\left( \rho ,\left( x,y\right)
\right) :=\Delta _{y}^{n}\rho \left( x\right) ,
\end{equation*}%
\begin{equation*}
H:Y\times Y\rightarrow Y\text{, \ }H\left( y_{1},y_{2}\right) :=\left(
n!\right) y_{2}.
\end{equation*}%
We note that equation (\ref{3}) becomes equation (\ref{5}) and its characteristic is equation (\ref{6}) or, the monomial equation%
\begin{equation*}
\Delta _{y}^{n}\rho \left( x\right) =\left( n!\right) \rho \left( y\right) 
\text{ for all }x,y\in X.
\end{equation*}

1. If $\rho \in \mathcal{M}_{n}\left( X,Y\right) $\ and $f:U\rightarrow Y$\
is defined by (\ref{9}), then $f=\rho \circ g$\ and - from Lemma 2.2.2 - $f$
is a solution of equation (\ref{3}).

2. Let $f\in \mathcal{S}_{3}$ and $u_{0}\in U$. Since $\left( Y,+\right) $\
is $\left( n!\right) $-divisible, the function%
\begin{equation*}
f\left( U\right) \rightarrow Y\text{, }y\mapsto H\left( f\left( u_{0}\right)
,y\right) =\left( n!\right) y\text{ }
\end{equation*}%
is an injection and, from Lemma 2.2.3, there exists an $n$-monomial $\rho \in 
\mathcal{M}_{n}\left( X,Y\right) $\ such that $f=\rho \circ g$; therefore $f$
satisfies\ relation (\ref{9}).
\end{proof}

Finally we characterize the solutions of equation (\ref{4}).

\begin{thm}
Let $f:U\rightarrow Y$\ be a function. Then $f$\ is a solution of equation (%
\ref{4}) if and only if there exists the monomials $\rho _{i}\in \mathcal{M}%
_{i}\left( X,Y\right) $\ for $i\in \left\{ 0,1,...,n\right\} $\ such that%
\begin{equation}
f\left( u\right) =\sum\limits_{i=0}^{n}\rho _{i}\left( \left( \ln
\left\vert \sin u\right\vert \right) \right) \text{ for all }u\in U\text{.}
\label{10}
\end{equation}
\end{thm}

\begin{proof}
Let%
\begin{equation*}
g:U\rightarrow X\text{, }g\left( u\right) :=\ln \left\vert \sin u\right\vert 
\text{ for all }u\in U\text{,}
\end{equation*}%
\begin{equation*}
g^{\prime }:X\rightarrow U\text{, }g^{\prime }\left( x\right) :=\arcsin e^{x}%
\text{ for all }x\in X\text{,}
\end{equation*}%
\begin{equation*}
G:Y^{X}\times X^{2}\rightarrow Y\text{, }G\left( \rho ,\left( x,y\right)
\right) :=\sum\limits_{i=1}^{n+1}\left( -1\right) ^{i}\tbinom{n+1}{i}\rho
\left( x+iy\right) \text{ and}
\end{equation*}%
\begin{equation*}
H:Y\times Y\rightarrow Y\text{, }H\left( y_{1},y_{2}\right) :=-y_{1}.
\end{equation*}%
We note that equation (\ref{4})\ can be rewritten in the form (\ref{5}) and, consequently, its characteristic is equation (\ref{6}) or, equivalent, the Fr\'{e}chet polynomial equation\ 
\begin{equation*}
\Delta _{y}^{n+1}\rho \left( x\right) =0\text{ for all }x,y\in X.
\end{equation*}%

Then, according to Lemma 2.2.2, $\mathcal{S}_{4}^{c}\subseteq \mathcal{S}_{4}$%
. Moreover, if $f\in \mathcal{S}_{4}$\ and $u_{0}$\ is an arbitrary number
in $U$, the function%
\begin{equation*}
f\left( U\right) \rightarrow Y\text{, }y\mapsto H\left( y,f\left(u_{0}\right)\right) =-y
\end{equation*}%
is bijective; from Lemma 2.2.3 we have $f\in \mathcal{S}_{4}^{c}$. Hence%
\begin{equation*}
\mathcal{S}_{4}=\left\{ \rho \circ g|\text{ }\rho \in \mathcal{P}_{n}\left(
X,Y\right) \right\} ,
\end{equation*}%
and, by Lemma 2.1, 
\begin{equation*}
\mathcal{P}_{n}\left( X,Y\right) =\mathcal{M}_{0}\left( X,Y\right) +\mathcal{%
M}_{1}\left( X,Y\right) +\cdots +\mathcal{M}_{n}\left( X,Y\right) .
\end{equation*}%
Consequently, if $f:U\rightarrow Y$, then $f\in \mathcal{S}_{4}$ if and only
if there exist $\rho _{i}\in \mathcal{M}_{i}\left( X,Y\right) $\ for $i\in
\left\{ 0,1,...,n\right\} $\ such that $f\left( u\right)
=\sum\limits_{i=0}^{n}\rho _{i}\left( \left( \ln \left\vert \sin
u\right\vert \right) \right) $ for all $u\in U$, and the theorem is
completely proved.
\end{proof}

\bigskip 
{\footnotesize
\hspace*{0.5cm}

\begin{minipage}[t]{8cm}$$\begin{array}{l}
\mbox{Dan M. Daianu -- Department of Mathematics,}\\
\mbox{'Politehnica' University of Timi\c soara},\\
 \mbox{P-ta Victoriei 2, 300 006, Timi\c soara, ROMANIA}\\
\mbox{E-mail: dan,daianu.m@gmail.com}\end{array}$$
\end{minipage}}

\end{document}